\documentclass[a4paper,10pt]{article}
\usepackage{amsfonts,amsmath,amssymb,amsthm,verbatim,placeins,comment,placeins,caption}
\usepackage[dvips]{graphicx}
\usepackage[square,sort,comma,numbers]{natbib}
\usepackage{bbm}
\usepackage{color}
\usepackage[usenames,dvipsnames,svgnames,table]{xcolor}
\usepackage{geometry}
\usepackage[cp1250]{inputenc}
\usepackage{lmodern}
\usepackage[OT2,T1]{fontenc}
\usepackage{graphicx}
\usepackage{caption,subfig}
\usepackage{fancyhdr}
\usepackage[style=english]{csquotes}

\newtheorem{theorem}{Theorem}[section]

\theoremstyle{definition}

\theoremstyle{remark}
\newtheorem{remark}[theorem]{Remark}

\numberwithin{equation}{section}

\newcommand{\N}{\mathbb{N}}
\newcommand{\R}{\mathbb{R}}

\begin{document}
\date{}
\title{\textbf{Correlated continuous time random walks and fractional Pearson diffusions}}

\maketitle

\centerline{\author{\textbf{N.N. Leonenko}$^{\textrm{a}}$, \textbf{I. Papi\'c}$^{\textrm{b}}$, \textbf{A. Sikorskii}$^{*\textrm{c}}$, \textbf{N. \v{S}uvak}$^{\textrm{b}}$}}

{\footnotesize{
$$\begin{tabular}{l}
  $^{\textrm{a}}$ \emph{School of Mathematics, Cardiff University, Senghennydd Road, Cardiff CF244AG, UK} \\
  $^{\textrm{b}}$ \emph{Department of Mathematics, J.J. Strossmayer University of Osijek, Trg Ljudevita Gaja 6, HR-31 000 Osijek, Croatia} \\
  $^{\textrm{c}}$ \emph{A423 Wells Hall, Department of Statistics and Probability, Michigan State University, East Lansing, MI 48824, USA} \\
\end{tabular}$$}}

\bigskip

\noindent\textbf{Abstract.}
Continuous time random walks have random waiting times between particle jumps. We define the correlated continuous time random walks (CTRWs) that converge to fractional Pearson diffusions (fPDs). The jumps in these CTRWs are obtained from Markov chains through the Bernoulli urn-scheme model and Wright-Fisher model. The jumps are correlated so that the limiting processes are not L\'evy but diffusion processes with non-independent increments. The waiting times are selected from the domain of attraction of a stable law.

\bigskip

\noindent\textbf{Key words.} Markov chains; fractional diffusion; Pearson diffusions; urn-scheme models; Wright-Fisher model; continuous time random walks

\bigskip

\noindent\textbf{Mathematics Subject Classification (2010):} 47D07, 60J10, 60J60, 60G22, 60G50.

\bigskip

\section{Introduction} \label{intro}

Continuous time random walks (CTRWs) have random waiting times between particle jumps. When jumps and waiting times are independent, the CTRW is called decoupled, and only decoupled CTRWs will be considered throughout this paper. For more information and applications of coupled and uncoupled CTRW we refer to \cite{Meerschaert2006114, PhysRevE.79.066102}.

When particle jumps $Y_1,\dots , Y_n,\dots$ are independent and identically distributed (iid), the random walk $S_{n}=Y_{1}+\ldots+Y_{n}$ converges to either the Brownian motion or a stable L\'evy process \cite[Chapter 4]{MeerschaertSikorskii_2011}. When jumps are separated by the iid waiting times $G_1,\dots G_n$ from the domain of attraction of a positively skewed stable law with stability index $0<\beta <1$, the process $S(N_t)$, where $T_{n}=G_{1}+\ldots+G_{n}$,  $N_{t}=\max\{n \geq 0 \colon T_{n} \leq t\}$, gives the location of a particle at time $t \geq 0$ and is known as the CTRW. By applying the continuous mapping theorem (see \cite{MeerschaertSikorskii_2011}, Theorem 4.19) it follows that, with proper scaling, $S(N(\lfloor ct \rfloor ))$ converges as $c \to \infty$, to $A(E_{t})$, where $A$ is either the Brownian motion or a stable L\'evy process, and $E_t$ is the inverse (or a hitting time) of a standard $\beta $-stable subordinator $\{D(t),\, t\ge 0\}$. Meerschaert and colleagues have shown the convergence to hold in $M_1$ and $J_1$ Skorokhod topology (\cite{meerschaert2004limit, Henry_Straka_2011}), and have obtained other L\'evy processes as the outer process in the limit by employing triangular arrays \cite{meerschaert2008triangular}. The case of correlated jumps given by the stationary linear process was considered in \cite{meerschaert2009correlated}, where the outer process in the limit was either a stable L\'evy process or a linear fractional stable motion, depending on the strength of the dependence in the particle jump sequence. In this paper, we define a different sequence of correlated particle jumps defined based on urn schemes that would result in weak convergence of the corresponding random walk to a Pearson diffusion. Pearson diffusions have stationary distributions of Pearson type. They include Ornstein-Uhlenbeck (OU), Cox-Ingersoll-Ross (CIR), Jacobi, Reciprocal Gamma, Fisher-Snedecor and Student diffusions (\cite{ALR_2009, AvramLeonenkoSuvak_WPD_2013}), and in this paper we consider the first three with non-heavy tailed stationary distributions. Then we separate jumps by the waiting times to obtain the CTRWs converging to fractional Pearson diffusions (fPDs).

Another approach to the limiting behaviour of CTRW is given in the recent paper \cite{kolokoltsov2009generalized} and book \cite{Kolokoltsov_2011} by Kolokoltsov, where the spatially non-homogeneous case is treated, i.e. the case in which the jump distribution depends on current particle position. To study the limiting behavior of such CTRWs, Kolokoltsov develops the theory of subordination of Markov processes by the hitting-time process, showing that this procedure leads to generalized fractional evolutions. Regarding the connection between scaling limits of CTRWs and Kolmogorov backward and forward equations we refer to \cite{baeumer2015fokker}.

The paper is organized as follows. We begin with a historical note on the subject of urn-scheme models based on a historically important paper \cite{markov1915problem} and a recent survey paper  \cite{jacobsen1996laplace} (heavily referring to \cite{Bernoulli_Laplace}), and a classical book \cite{Karlin_Taylor2}.  Section \ref{PD} gives the short overview of the general theory on Pearson diffusions, followed by Section \ref{FPD} on fractional Pearson diffusions, focusing on the non-heavy tailed cases. Finally, Section \ref{CTRWforFPD} is dedicated to generalizations of results from \cite{jacobsen1996laplace} and \cite{Karlin_Taylor2} and constructing the CTRWs converging to fractional Ornstein-Uhlenbeck, Cox-Ingersoll-Ross or Jacobi diffusions using the following steps:
\begin{enumerate}
 \item Obtaining the non-fractional Pearson diffusion $(X_{t}, \, t \geq 0)$ as the scaling limit of a suitably chosen or constructed  discrete time Markov chain.
\item Defining the corresponding CTRW process by taking the composition of the Markov chain defined in step 1 and the process $N_{t}=\max\{n \geq 0 \colon T_{n} \leq t\}$, modeling the number of jumps up to time $t \geq 0$.
\item  Proving that the composition converges to the corresponding fPD $\left(X(E_t),\,t\ge 0 \right)$, a composition of non-fractional Pearson diffusion with the inverse $(E_{t}, \, t \geq 0)$ of the $\beta$-stable subordinator with $\beta \in ( 0, 1 )$.
\end{enumerate}

\section{Historical roots} \label{history}

Many processes observed in science can be mathematically described by the urn-scheme models, where a discrete time Markov chain moves from one state to another. One of the most famous urn-scheme models is the Wright-Fisher model for gene mutations in a population under various assumptions. A rich class of the limiting processes could be obtained from it. One of the simplest urn-scheme models is the Bernoulli-Laplace urn model, named after D. Bernoulli and P. S. Laplace, the pioneers of probability theory, and later studied in depth by \cite{markov1915problem}. This model considers two urns, urn $A$ containing $j$ balls and urn $B$ containing $k$ balls. Of the total $j+k$ balls in two urns, suppose that $r$ balls are white, and $(j+k-r)$ are black. At time $n \in \mathbb{N}$ one ball is drawn randomly from urn $A$ and another from urn $B$. The ball drawn from urn $A$ is placed into urn $B$, the ball drawn from urn $B$ is placed into urn $A$. Let $X_n$ be the number of white balls in urn $B$ at time $n \in \N$. Then $(k-X_n)$ is the number of black balls in urn $B$, $(r-X_n)$ is the number of white balls in urn $A$ and $(j-r+X_n)$ is the number of black balls in urn $A$ at time $n \in \N$. Therefore, $(X_n, \, n \in \mathbb{N})$ is a homogeneous discrete time Markov chain with state space $S=\{\max{\{0, r-j\}}, \ldots, \min{\{k, r\}\}}$ and the following transition probabilities:
\begin{equation*}
P \left( X_{n+1} = y \vert X_{n} = x \right) = p_{x, y} = \left\{ \begin{array}{lcl} \displaystyle\frac{(r-x)(k-x)}{jk} & , & y = x+1 \\
                                                                                     \displaystyle\frac{(r-x)x+(j-r+x)(k-x)}{jk} & , & y = x \\
                                                                                     \displaystyle\frac{(j-r+x)x}{jk} & , & y = x-1 \\
                                                                                     0 & , & \rm{otherwise.}  \end{array} \right.
\end{equation*}

In his book \cite{Bernoulli_Laplace}, Laplace worked with a particular case of this model, in which each urn contains $n$ balls and also $n$ out of total $2n$ balls are black. In this setting he defined a discrete time Markov chain $\left( Z^{(n)}_{r}, \, r \geq 0 \right)$ with the state space $\{0, 1, 2, \ldots, n\}$, where $n$ is the number of balls in urn $A$ and $r$ is the number of draws, with transition probabilities
\begin{equation*}
p_{x,x+1}=\left(1-\frac{x}{n} \right)^2, \quad p_{x,x}= 2 \frac{x}{n}\left(1-\frac{x}{n}\right), \quad p_{x,x-1}=\left(\frac{x}{n} \right)^2 \quad \text{and $0$ otherwise}. \end{equation*}

Laplace was interested in finding the limiting distribution of this irreducible, aperiodic, reversible and ergodic  Markov chain with stationary/ergodic distribution
$$\pi = \{\pi_{0}, \cdots, \pi_{n} \}, \quad \pi_i = \frac{\binom {n} {i}\binom {n} {n-i}}{\binom {2n} {n}}.$$ By denoting $z_{x, \, r}$ the probability that there are $x$ white balls in urn $A$ after $r$ draws, he deduced the following partial second-order difference equation:
\begin{equation}\label{Laplace_diff_eq}
z_{x,\,r+1}=\left(\frac{x+1}{n} \right)^2 z_{x+1,\,r}+2 \frac{x}{n}\left(1-\frac{x}{n}\right) z_{x,\,r}+\left( 1-\frac{x-1}{n} \right)z_{x-1,\,r}.
\end{equation}

After that, instead of determining the generating function for $z$, he approximated the solution of the equation \eqref{Laplace_diff_eq} by introducing the new space variable $\mu$ and new time variable $r'$ and obtained the following relation: $$x=\frac{1}{2}(n+\mu \sqrt{n}), \quad r=nr'.$$

Then, he introduced the function $$U(\mu, \,r'):=z_{x, \, r}$$ of space and time and heuristically deduced the following relations describing the changes of state in the transformed Markov chain:
\begin{align*}
z_{x+1,\,r}&=z_{x,\,r}+\frac{\partial z_{x,\,r}}{\partial x}+\frac{1}{2}\frac{\partial^2 z_{x,\,r}}{\partial x^2}\\
z_{x-1,\,r}&=z_{x,\,r}-\frac{\partial z_{x,\,r}}{\partial x}+\frac{1}{2}\frac{\partial^2 z_{x,\,r}}{\partial x^2}\\
z_{x,\,r+1}&=z_{x,\,r}+\frac{\partial z_{x,\,r}}{\partial r}.
\end{align*}

By another purely heuristic argument, he claimed that function $U(\mu, \,r')$ satisfies the second-order differential equation
\begin{equation}\label{OU_KFE}
\frac{\partial U}{\partial r'}=-\frac{\partial }{\partial \mu}(-2 \mu U)+\frac{1}{2}\frac{\partial^2 }{\partial \mu^2}(2 U)=2U+2\mu \frac{\partial U}{\partial \mu}+ \frac{\partial^2 U}{\partial \mu^2},
\end{equation}
which is a special case of the Fokker-Planck or Kolmogorov forward equation defining the OU diffusion:
\begin{equation*}
        \frac{\partial p(x,t)}{\partial t}= -\frac{\partial}{\partial x}\left(\mu(x) p(x,t) \right)+\frac{1}{2}\frac{\partial^2}{\partial x^2}\left(
        \sigma^2(x) p(x,t) \right).
        \end{equation*}
In particular, equation \eqref{OU_KFE} is the governing equation for the OU diffusion with infinitesimal mean $\mu(x)=-2x$ and infinitesimal variance $\sigma^2(x)=2$. While rigorous proofs were not provided, this work had the first mention of the forward equation for the OU diffusion.

A century later, Markov \cite{markov1915problem} considered a more general model. In his scheme there are also two urns, urn $A$ containing $n$ balls and urn $B$ containing $n_{1}$ balls. Out of total $(n+n_1)$ balls, there are $(n+n_1)p$ white and $(n+n_1)q$ black balls ($0 < p < 1, q = 1 - p$). Similar to the Laplace's scheme, by denoting probability that there are $x$ white balls in the urn $A$ after $r$ draws by $z_{x, \, r}$, Markov obtained the following difference equation:
\begin{align}\label{difference_Markov}
z_{x,\,r+1}=&\frac{x+1}{n}\cdot\frac{n_1q-np+x+1}{n_1}\cdot z_{x+1,\,r}+\frac{n-x+1}{n}\cdot\frac{(n+n_1)p-x+1}{n_1}\cdot z_{x-1,\,r}\nonumber\\&+\left (\frac{x}{n}\cdot\frac{(n+n_1)p-x}{n_1}+\frac{n-x}{n}\cdot\frac{n_1q-np+x}{n_1} \right )\cdot z_{x,\,r}.
\end{align}

Next, Markov introduced new space variable $\mu$ and new time variable $\rho$ and obtained the following relation:
$$x=np+\mu \frac{1}{\Delta\mu},\,\,r \left( \frac{1}{n}+\frac{1}{n_1} \right )=2\rho, \,\,\text{where}\,\, \Delta \mu =\sqrt{\frac{n+n_1}{2 p q n n_1}},$$
demanding that the ratio of number of balls in urn $A$ and urn $B$ remains constant at all times, i.e. $n_1 \alpha = n$ for some $\alpha > 0$. Obviously, for $\alpha=1$ and $p = q = \frac{1}{2}$ this model reduces to the previously described Laplace's urn scheme. 

Finally, the difference equation \eqref{difference_Markov} was, again by pure heuristics, approximated by the corresponding Fokker-Planck equation. In order to do so, Markov defined the space and time dependent function $U(\mu, \rho):=z_{x, r}$, smooth enough for obtaining approximations
\begin{align*}
z_{x+1,\,r}&=U(\mu+\Delta \mu, \, \rho)\approx U+\Delta \mu\frac{\partial U}{\partial \mu}+\frac{1}{2}(\Delta\mu)^2\frac{\partial^2 U}{\partial \mu^2}\\
z_{x-1,\,r}&=U(\mu-\Delta \mu, \, \rho)\approx U-\Delta \mu\frac{\partial U}{\partial \mu}+\frac{1}{2}(\Delta\mu)^2\frac{\partial^2 U}{\partial \mu^2}\\
z_{x,\,r+1}&=U(\mu, \, \rho+\frac{1+\alpha}{2n}) \approx U+\frac{1+\alpha}{2n}\frac{\partial U}{\partial \rho},
\end{align*}
precise up to the order $o(n^{-1})$.

By combining these approximations with the new space and time transformations and putting it into \eqref{difference_Markov}, Markov obtained the second order differential equation
\begin{equation*}
\frac{\partial U}{\partial \rho}=2U+2\mu \frac{\partial U}{\partial \mu}+ \frac{\partial^2 U}{\partial \mu^2},
\end{equation*}
the Fokker-Planck equation for the OU diffusion with infinitesimal mean $\mu(x) =-2x$ and infinitesimal variance $\sigma^2(x) = 2$, completely coinciding with Laplace's result.

We now briefly discuss the Wright-Fisher urn scheme, named after S. Wright and R. Fisher. Wright-Fisher urn scheme is a model that describes gene mutations (in some genetic pool) over time, strongly influencing selection and sampling forces in the corresponding population. There are several different versions of this model, and we use the scheme described in \cite{Karlin_Taylor2}.

Suppose that in a population of size $n$  each individual is either of type $A$ or type $a$. Let $i$ be the number of $A$-types in the population. Therefore, the remaining $(n-i)$ population members are $a$-types.
The next generation of a population is produced depending on the influence of mutation, selection and sampling forces. Once born, individual of $A$-type can mutate in $a$-type with probability $\alpha$ and individual of $a$-type can mutate in $A$-type with probability $\beta$. Taking into account parental population comprised of $i$ $A$-types and $(n-i)$ $a$-types, the expected fraction of $A$-types after mutation is
\begin{equation*}
  \frac{i}{n}\left(1-\alpha \right)+\left(1-\frac{i}{n}\right)\beta,
\end{equation*}
while the expected fraction of $a$-types is
\begin{equation*}
  \frac{i}{n}\alpha +\left(1-\frac{i}{n}\right)\left(1-\beta \right).
\end{equation*}

Survival ability of each type is modeled by parameter $s$ so that the ratio of $A$-types over $a$-types is equal to $1+s$, meaning that $A$-type is selectively superior to $a$-type. Then the expected fraction of mature $A$-types before reproduction is
\begin{equation}\label{WF_pi}
p_i=\frac{\left(1+s\right)\left[i\left(1-\alpha \right)+\left(n-i\right)\beta\right]}{\left(1+s\right)\left[i\left(1-\alpha \right)+\left(n-i\right)\beta\right]+\left[i\alpha +\left(n-i\right)\left(1-\beta \right)\right]}.
\end{equation}
The last assumption of this model is that the composition of the next generation is determined through $n$ binomial trials, where the probability of producing an $A$-type in each trial is $p_i$. This model, tracking the number of $A$-types in population over time, can be described by the discrete-time Markov chain $(G^{n}_{r}, \, r \in \mathbb{N}_{0})$ with the state space $ \{0, 1, 2, \ldots, n \}$ and transition probabilities
\begin{equation}\label{WF_pij}
p_{ij} = \binom{n}{j} p^{j}_{i} (1-p_i)^{n-j}.
\end{equation}

This model is parametrized by $\alpha $, $\beta$,  and $s \in [0,1]$. Depending on their values, different limiting diffusions could be obtained. In Section \ref{CTRWforFPD_WFUS} of this paper, we present two different settings that lead to generally parametrized Jacobi and CIR diffusions. The procedures of constructing these diffusions as limits of some suitably selected discrete-time Markov chain (urn scheme) are based on examples given in \cite[p. 176-183]{Karlin_Taylor2}, where only heuristic arguments are given for only special cases of the limiting Jacobi and CIR diffusions.

From this review of history, we see that the connection between discrete-time Markov chains  in  urn-scheme models and some limiting continuous time stochastic processes has been brought up in the literature. Today, the conjectures of Laplace and Markov could be rigorously proved by modern techniques of analysis and probability theory, involving the convergence of evolution operators of discrete time Markov chains. In this paper we derive such connections between several urn-scheme models and the corresponding diffusions. More precisely, for appropriately chosen discrete time Markov chains we derive the corresponding limiting diffusions, then define the CTRWs and their fractional non-heavy tailed Pearson diffusion limits.

\section{Pearson diffusions}\label{PD}

In 1931 Kolmogorov observed that the differential equation for the invariant
density $\mathfrak{m}(x),x\in \mathbb{R}$, of the classical Markovian
diffusions satisfying
\[
\mathfrak{m}^{\prime }(x)/\mathfrak{m}(x)=[a(x)-b^{\prime
}(x)]/[b(x)]=[(a_{1}-2b_{2})x+(a_{0}-b_{1})]/[b_{2}x^{2}+b_{1}x+b_{0}],
\]%
in case of a linear drift $a(x)=a_{1}x+a_{0}=-\theta (x-\mu )$
and the quadratic squared diffusion coefficient $\sigma ^{2}(x)=2\theta
b(x), $ becomes the famous Pearson equation introduced by Pearson in 1914 to unify
important families of distributions (see \cite{PearsonTables_1914}). The family of Pearson distributions is
categorized into six well known and widely applied parametric subfamilies:
normal, gamma, beta, Fisher-Snedecor, reciprocal gamma and Student
distributions. The latter three distributions are heavy-tailed, while the former three are not.
Pearson diffusion is defined as the unique strong solution of the following
stochastic differential equation (SDE):
\[
\label{SDE} dX_{t}=-\theta (X_{t}-\mu )dt+\sqrt{2\theta
(b_{2}X_{t}^{2}+b_{1}X_{t}+b_{0})}dW_{t},\,\,t\geq 0,
\]%
where $\mu \in \mathbb{R}$ is the stationary mean, $\theta >0$ is the
scaling of time determining the speed of the mean reversion, and $b_0$, $b_1$
and $b_2$ must be such that the square root is well defined when $X_{t}$ is in
the state space $(l,L)$.

Pearson diffusions could be categorized into six subfamilies (see \cite{LeonenkoMeerschaertSikorskii_FPD_2013}),
according to the degree of the polynomial $\sigma ^{2}(x)$ and, in the
quadratic case, according to the sign of its leading coefficient $b_2$ and the
sign of its discriminant $\Delta =b_{1}^{2}-4b_{0}b_{2}$. The six types are:

\begin{itemize}
\item constant $b(x)$ - Ornstein-Uhlenbeck (OU) process, characterized by a
normal stationary distribution,

\item linear $b(x)$ - Cox-Ingersol-Ross (CIR) process, characterized by a
gamma stationary distribution,

\item quadratic $b(x)$ with $b_{2}<0$ - Jacobi (JC) diffusion, characterized
by a beta stationary distribution,

\item quadratic $b(x)$ with $b_{2}>0$ and $\Delta (b)>0$ - Fisher-Snedecor
(FS) diffusion, characterized by Fisher-Snedecor stationary distribution,

\item quadratic $b(x)$ with $b_{2}>0$ and $\Delta (b)=0$ - reciprocal gamma
(RG) diffusion, characterized by reciprocal gamma stationary distribution,

\item quadratic $b(x)$ with $b_{2}>0$ and $\Delta (b)<0$ - Student (ST)
diffusion, characterized by Student stationary distribution.
\end{itemize}

More details on the first three types of Pearson diffusions with non-heavy
tailed stationary distributions are found in the classical book \cite{Karlin_Taylor2}.

Time evolution of the diffusion is described by the Kolmogorov forward
(Fokker-Planck) and backward partial differential equations (PDEs) for the
transition density $p(x,t;y,s)=\frac{d}{dx}\mathit{P}(X_{t}\leq x|X_{s}=y)$.
We consider only time-homogeneous diffusions for which $%
p(x,t;y,s)=p(x,t-s;y,0)$ for $t>s$ and thus we write $p(x,t;y)=\frac{d}{dx}%
\mathit{P}(X_{t}\leq x|X_{0}=y).$ The transition density $p(x,t;y)$ solves
the following Cauchy problem with space-varying polynomial coefficients:
\[
\label{backward} \frac{\partial p(x,t;y)}{\partial t}=\mathcal{G}%
p(x,t;y),\quad p(x,0;y)=g(y)
\]
with the point-source initial condition, where $\mathcal{G}g(y)=\left( \mu
(y)\frac{\partial }{\partial y}+\frac{\sigma ^{2}(y)}{2}\frac{\partial ^{2}}{%
\partial y^{2}}\right) g(y)$ is the infinitesimal generator of the
diffusion, playing a key role in its analytical properties. The generators
of non-heavy tailed OU, CIR, and Jacobi diffusions considered in this paper have purely discrete
spectra, consisting of infinitely many simple eigenvalues $(\lambda
_{n},\,n\in \mathbb{N})$ (see \cite{LeonenkoMeerschaertSikorskii_FPD_2013}). Corresponding orthonormal
eigenfunctions $(Q_{n}(x),n\in \mathbb{N})$ are classical systems of
orthogonal polynomials: Hermite polynomials for the OU process, Laguerre
polynomials for the CIR process and Jacobi polynomials for the Jacobi diffusion.

 \section{Fractional Pearson diffusions (fPDs)}\label{FPD}

Let $(X_{1}(t),\,t\geq 0)$ be a Pearson diffusion. The fPD $(X_{\beta
}(t),\,t\geq 0)$ is defined via a non-Markovian time-change $E(t)$
independent of $X_{1}(t)$:
\[
X_{\beta }(t):=X_{1}(E_{t}),\quad t\geq 0,
\]%
where $E_{t}=\inf \{x>0:D_{x}>t\}$ is the inverse of the standard $\beta $%
-stable L\'{e}vy subordinator $(D_{t},\,t\geq 0)$, $0<\beta <1$, with the
Laplace transform $\mathit{E}[e^{-sD_{t}}]=\exp \{-ts^{\beta }\}, s\geq
0$. Since $E_{t}$ rests for periods of time with non-exponential
distribution, the process $(X_{\beta }(t),\,t\geq 0)$ is not a Markov
process.

We say that the non-Markovian process $X_{\beta }(t)$ has a transition
density $p_{\beta }(x,t;y)$ if
\[
\mathit{P}(X_{\beta }(t)\in B|X_{\beta }(0)=y)=\int_{B}p_{\beta }(x,t;y)dx
\]%
for any Borel subset $B$ of $(l,L)$.

The governing equations for the fPDs are time-fractional forward and
backward Kolmogorov equations, with the time-fractional derivative
(regularized non-local operator) defined in the Caputo sense (see \cite{MeerschaertSikorskii_2011}):
\[
\frac{\partial ^{{}\beta }u}{\partial t^{{}\beta }}=\left\{
\begin{array}{ll}
\frac{\partial u}{\partial t}\left( t,x\right) , & \text{if }\beta =1 \\
\frac{1}{\Gamma \left( 1-\beta \right) }\frac{\partial }{\partial t}%
\int_{0}^{t}\left( t-\tau \right) ^{-\beta }u\left( \tau ,x\right) d\tau -%
\frac{u(0,x)}{t^{\beta }}, & \text{if}\ \beta \in (0,1).%
\end{array}%
\right.
\]

Using spectral methods, the representations for the transition densities of
the non-heavy-tailed fPDs (OU, CIR, Jacobi) were obtained in \cite{LeonenkoMeerschaertSikorskii_FPD_2013}.
Namely, it has been shown that the series
\begin{equation}
p_{\beta }(x,t;y)=\mathbf{m}(x)\sum_{n=0}^{\infty }\mathcal{E}_{\beta
}\left( -\lambda _{n}t^{\beta }\right) Q_{n}(y)Q_{n}(x)  \label{eq16}
\end{equation}%
converges for fixed $t>0$, $x,\,y\in (l,L)$, where $\mathcal{E}_{\beta
}(-z)=\sum\limits_{j=0}^{\infty }(-z)^{j}/\Gamma (1+\beta j),z\geq 0,$ is
the Mittag-Leffler function. The series \eqref{eq16} satisfies $p_{\beta
}(x,0;y)=\delta (x-y)$. These spectral representations were then used to
obtain the explicit strong solutions of the corresponding fractional Cauchy
problems for both backward and forward equations. The solutions were given as series that converge and satisfy the fractional partial differential equations pointwise. The results obtained in this paper can be used to implement another method of obtaining the solutions for the fractional Cauchy problems with space-varying polynomial coefficients appearing in the generator. Namely, by simulating the Markov chains and waiting times, particle tracking method can provide the means for numerically evaluating the strong solutions.

\section{Transition operators of the discrete time Markov chains} \label{CTRWforFPD}

To implement the diffusion approximation of discrete-time Markov chains, we will use transition kernels and transition operators. Let $\mu$ be an arbitrary probability kernel on a measurable space $(S, \mathcal{S})$. The associated transition operator $T$ is defined as
\begin{equation}
\label{TransOp}
  Tf(x) = (Tf)(x) = \int \mu(x, dy) f(y), \quad x \in S,
\end{equation}
where $f \colon S \to \mathbb{R}$ is assumed to be measurable and either bounded or nonnegative. From the approximation of  $f$ by simple functions and the monotone convergence argument, it follows
that $Tf$ is again a measurable function on $S$. Furthermore, $T$ is a positive contraction operator: $0 \leq f \leq 1$ implies that $0 \leq Tf \leq 1$. For more details on transition operators and their importance to the study of Markov processes, we refer to \cite{Kallenberg_2002}, Chapter 19.

Theorems 19.25 and 19.28 from \cite{Kallenberg_2002} are crucial for the technique used for obtaining the non-fractional Pearson diffusion as the scaling limit of a suitably chosen Markov chain with known transition operator. Markov chain will be defined in terms of its state space and  transition probabilities (the integral in \eqref{TransOp} then becomes a sum). A few additional definitions are given for clarity before the statement of the theorem.

By $\mathbb{D}(S)$ we denote the space of right continuous functions with left limits defined on $\R^{+}$ with values in $S$ endowed with Skorokhod $J_1$ topology.
Consider the Banach space of bounded continuous functions on space $S$, where  $S=\R$ in the OU case, $S=[0,1]$ in the Jacobi case and $S=[0,+\infty)$ in  the CIR case, with the supremum norm.
For a closed operator $\mathcal{A}$ with domain $\mathcal{D}$, a core for $\mathcal{A}$ is a linear subspace $D \subset \mathcal{D}$ such that the restriction $\mathcal{A} \vert_{D}$ has closure $\mathcal{A}$. In that case, $\mathcal{A}$ is clearly uniquely determined by its restriction $\mathcal{A} \vert_{D}$.

Theorems 1.6 and 2.1 from \cite[Section 8]{ethier2009markov} give sufficient conditions for $C^{\infty}_c(S)$ being a core of the diffusion infinitesimal generator. In particular, Jacobi diffusion satisfies conditions of Theorem 1.6, while other Pearson diffusions satisfy conditions of Theorem 2.1, which means $C^{\infty}_c(S)$ is a core for Pearson diffusions. Therefore $C^{3}_c(S)$, as a broader space, can be referred to as the core as well.

\begin{theorem}\label{chain_aprox}
Let $(Y^{(n)}, \, n \in \mathbb{N})$ be a sequence of discrete-time Markov chains on $S$ with transition operators $(U_{n}, \, n \in \mathbb{N})$. Consider a Feller process $X$ on $S$ with semigroup $T_t$ and generator $\mathcal{A}$. Fix a core $D$ for the generator $\mathcal{A}$, and assume that $(h_n, \, n \in \mathbb{N})$ is the sequence of positive reals tending to zero as $n \to \infty$. Let
\begin{equation*}
  A_n = h^{-1}_n(U_n-I), \quad T_{n,t} = U^{\lfloor t/h_n \rfloor}_{n}, \quad X^{n}_{t} = Y^{n}_{\lfloor t/h_n \rfloor}.
\end{equation*}
Then the following statements are equivalent:
\begin{itemize}
\item[a)] If $f \in D$, there exist some $f_n \in Dom(A_n)$ with $f_n \to f$ and $A_n f_n \to \mathcal{A}f$ as $n \to \infty$
\item[b)] $T_{n,t} \to T_t$ strongly for each $t>0$
\item[c)] $T_{n,t}f_n \to T_t f$ for each $f \in C_0$, uniformly for bounded $t>0$
\item[d)] if $X^{(n)}_{0}\Rightarrow X_{0}$ in $S$, then $X^{n}\Rightarrow X$ in the Skorokhod space $\mathbb{D}(S)$ with the $J_{1}$ topology.
\end{itemize}
\end{theorem}
The proof could be found in \cite[Theorem 19.28, page 387]{Kallenberg_2002}.

\section{Correlated CTRW for the OU process based on the Laplace-Bernoulli urn scheme} \label{CTRWforFPD_BLUS}

In this section, we define the Bernoulli-Laplace urn-scheme model, but compared to early ideas of Laplace and Markov summarized in Section \ref{history}, with crucial changes in space and time transformations in order to obtain more general limiting process, i.e. the OU diffusion with general parameters. The urn scheme consists of two urns, $A$ and $B$, each containing $n$ balls. Furthermore, $n$ out of total $2n$ balls are black. At each step one ball is randomly chosen from each urn. The ball drawn from urn $A$ is then placed into urn $B$ and the ball drawn from urn $B$ is placed into urn $A$. We are interested in the number of white balls in the urn $A$ after $r \in \mathbb{N}$ draws.

Let for each $n \in \mathbb{N}$, $(Z^{(n)}_{r}, \, r \in \mathbb{N}_{0})$ be the Markov chain with the state space $ \{0, 1, 2, \ldots, n\}$, where $n$ is the number of balls in urn $A$, $r$ is the number of draws and  $Z^{(n)}_{r}$ is the number of white balls in the urn $A$ after $r$ draws. The transition probabilities for this Markov chain are as follows:
\begin{equation} \label{trans}
p_{i, i+1} = \left(1 - \frac{i}{n} \right)^2, \quad p_{i, i} = 2 \frac{i}{n} \left(1 - \frac{i}{n}\right), \quad p_{i, i-1} = \left(\frac{i}{n} \right)^2, \quad \text{$0$ otherwise}.
\end{equation}

We assume that the initial state of the chain is given by $Z^{(n)}_{0} = \lfloor \frac{1}{2} (n+(ax+b) \sqrt{n}) \rfloor$, where $a \neq 0$ and $b$  are fixed parameters with values in $\R$ and $x\in \R$. Now we introduce the space variable (i.e., state) transformation
$$i' = \frac{1}{a\sqrt{n}} \left ( 2i - n - b\sqrt{n} \right ), \quad i \in \{0, 1, 2, \ldots, n\},$$
to obtain a new Markov chain $(H^{(n)}_{r}, \, r \in \mathbb{N}_{0})$, where
\begin{equation}\label{new_chain_OU}
H^{(n)}_{r} = \frac{1}{a \sqrt{n}} \left ( 2Z^{(n)}_{r} - n - b \sqrt{n} \right).
\end{equation}

Let us denote $x_n=\lfloor \frac{1}{2} (n+(ax+b)\sqrt{n}) \rfloor$ and $x'_n= \frac{1}{a\sqrt{n}}\left ( 2 x_n - n -b \sqrt{n} \right ) $. Since our starting Markov chain $(Z^{(n)}_{r}, \, r \in \mathbb{N}_{0})$ is homogenous, using \eqref{trans} we obtain the transition operator of the Markov chain $(H^{(n)}_{r}, \, r \in \mathbb{N}_{0})$:
\begin{align}\label{trans_op_OU}
T_n f(x'_n)&=\left(1-\frac{1}{2}\left(1+\frac{ax'_n+b}{\sqrt{n}}\right) \right)^2 f\left(x_n'+\frac{2}{a\sqrt{n}}\right)+  \left(1+\frac{ax'_n+b}{\sqrt{n}}\right)\left(1-\frac{1}{2}\left(1+\frac{ax'_n+b}{\sqrt{n}}\right)\right)f(x'_n)\nonumber\\
&+\left(\frac{1}{2}\left(1+\frac{ax'_n+b}{\sqrt{n}}\right ) \right)^2 f\left(x'_n-\frac{2}{a\sqrt{n}}\right)\nonumber\\
&=\left(1-\frac{x_n}{n} \right)^2 f\left(\frac{1}{a\sqrt{n}}\left ( 2 (x_n+1) - n -b \sqrt{n} \right )\right)+ 2 \frac{x_n}{n}\left(1-\frac{x_n}{n}\right)f\left(\frac{1}{a\sqrt{n}}\left ( 2 x_n - n -b \sqrt{n} \right )\right)\nonumber\\
&+\left(\frac{x_n}{n} \right)^2 f\left(\frac{1}{a\sqrt{n}}\left ( 2 (x_n-1) - n -b \sqrt{n} \right )\right).
\end{align}

For the application of   Theorem \ref{chain_aprox}, operators $A_{n}$ are defined as follows:
\begin{equation}\label{CTRW_gen_OU}
A_n: = \frac{\theta}{2} n(T_n-I),\quad \theta >0, \quad f_n \in Dom(A_n), \quad f_n(x):=f(x'_n).
\end{equation}
where $f \in C_{c}^{3}(\R).$ \newline
The continuous-time process $(X^{(n)}_t, \, t \geq 0)$ is obtained by the following scaling of time in $(H^{(n)}_{r}, \, r \in \mathbb{N}_{0})$, for each $n \in \mathbb{N}$:
\begin{equation}\label{TimeChange_OU}
X^{(n)}_t := H^{(n)}_{\lfloor \frac{\theta}{2} nt \rfloor}, \,\, \theta >0.
\end{equation}

Let $X = (X_t, \, t \geq 0)$ be the OU diffusion, i.e., the solution of the SDE
\begin{equation*}
dX_t= \mu(X_t)dt+\sigma(X_t)dW(t), \quad t \geq 0,
\end{equation*}
where $(W(t), t\geq 0)$ is the standard Brownian motion. The drift $\mu(x)$ and diffusion parameter $\sigma^2(x)$ are given by
$$\mu(x)=-\tau \left (x-\mu \right ), \quad \sigma(x)=\sqrt{2\tau\sigma^2}, \quad \tau>0, \quad \mu \in \R, \quad \sigma \in \R,$$
and the infinitesimal generator is defined by the following expression:
\begin{equation}\label{OUgenerator}
\mathcal{A}f(x)=\left [\mu(x)\frac{\partial}{\partial x}+\frac{\sigma^2(x)}{2}\frac{\partial^2}{\partial x^2}\right ]f(x)=\left [-\tau \left (x-\mu \right )\frac{\partial}{\partial x}+\tau \sigma^2\frac{\partial^2}{\partial x^2}\right]f(x),\quad f \in C^{3}_{c}(\R).
\end{equation}

The next theorem established that the OU diffusion is the limiting process of the sequence of rescaled discrete time Markov chains given by \eqref{new_chain_OU}.

\begin{theorem}\label{Thm_OUDiffusion}
Let $(H^{(n)}_{r}, \, r \in \mathbb{N}_{0})$ for each $n \in \mathbb{N}$,  be the Markov chain defined by \eqref{new_chain_OU} with transition operator \eqref{trans_op_OU}. Let $(X^{(n)}_t, \, t \geq 0)$ for each $n \in \N$, be the time-changed process corresponding to the Markov chain $(H^{(n)}_{r}, \, r \in \mathbb{N}_{0})$, with the time-change  \eqref{TimeChange_OU}. Let operators $(A_n, \, n \in \N)$ be defined by \eqref{CTRW_gen_OU}.
Then
$$X^{n} \Rightarrow X \,\,\text{in}\,\, \mathbb{D}(\R),$$
where  $X = (X_t, \, t \geq 0)$ is the Ornstein-Uhlenbeck diffusion with infinitesimal generator $\mathcal{A}$ given by \eqref{OUgenerator} and parameters $\tau=\theta$, $\mu=-\frac{b}{a}$ and $\sigma=\frac{1}{a\sqrt{2}}$.
\end{theorem}

\begin{proof}
In the setup of the Theorem \ref{chain_aprox}, we will first prove that statement $a)$ of this theorem is valid for $A_n$ defined by \eqref{CTRW_gen_OU} and $\mathcal{A}$ defined by \eqref{OUgenerator}. Then we will use the equivalence of statements $a)$ and $d)$ to obtain convergence $X^{n} \Rightarrow X$ in $\mathbb{D}(\R)$ under assumption $X^{(n)}_{0}\Rightarrow X_{0}$, $n \to \infty$.

Therefore, first we prove the validity of the statement $a)$ from Theorem \ref{chain_aprox} in our setting:

$$||f_n-f||_{\infty}=\underset{x \in \R}{sup}\, |f_n(x)-f(x)|=\underset{x \in \R}{sup}\, |f(x'_{n})-f(x)|$$
and since by the definition of function $\lfloor \cdot \rfloor$
$$a x'_n \leq a x \le a x'_n +\frac{2}{\sqrt{n}}$$
we have
$$x'_n \rightarrow x, \,\, n \rightarrow \infty,$$
so using this we obtain
$$||f_n-f||_{\infty} \rightarrow 0, \,\, n \rightarrow \infty$$

and
$$||A_n f_n-\mathcal{A} f||_{\infty}=\underset{x \in \R}{sup} \,|A_n f_n(x)-\mathcal{A} f(x)|=\underset{x \in \R}{sup}\, |A_n f(x'_{n})-\mathcal{A} f(x)|.$$

By the triangle inequality we have
\begin{equation} \label{OU_gen_conver}
||A_n f_n-\mathcal{A} f||_{\infty} \leq \underset{x \in \R}{sup}\, |A_n f(x'_{n})-A_n f(x)| + \underset{x \in \R}{sup}\, |A_n f(x)-\mathcal{A} f(x)|.
\end{equation}

Now it follows that
\begin{align*}
A_n f(x)&=\frac{\theta}{2} n\left (T_n f(x)-f(x)\right)\\
        &=\frac{\theta}{2} n \left(1-\frac{1}{2}\left(1+\frac{ax+b}{\sqrt{n}}\right) \right)^2 f\left(x+\frac{2}{a\sqrt{n}}\right)+ \frac{\theta}{2} n\left(1+\frac{ax+b}{\sqrt{n}}\right)\left(1-\frac{1}{2}\left(1+\frac{ax+b}{\sqrt{n}}\right)\right)f\left(x\right)\\
        &+ \frac{\theta}{2} n\left(\frac{1}{2}\left(1+\frac{ax+b}{\sqrt{n}}\right ) \right)^2 f\left(x-\frac{2}{a\sqrt{n}}\right) - \frac{\theta}{2} n f\left(x\right)\\
        &=\frac{\theta}{2} n\left ( f\left(x+\frac{2}{a\sqrt{n}} \right)-f\left(x\right) \right )-\frac{\theta}{2} n\left (1+\frac{ax+b}{\sqrt{n}} \right )\left( f\left(x+\frac{2}{a\sqrt{n}}\right)-f\left ( x \right )\right)\\
        &+ \frac{\theta n}{8}\left (\left(1+\frac{ax+b}{\sqrt{n}}\right)\right )^2\left ( f\left(x+\frac{2}{a\sqrt{n}}\right)-2f\left(x\right)+f\left(x-\frac{2}{a\sqrt{n}}\right)\right )\\
        &= -\theta\left(x+\frac{b}{a}\right)\frac{\left ( f\left(x+\frac{2}{a\sqrt{n}} \right)-f\left(x\right) \right )}{\frac{2}{a\sqrt{n}}}+\frac{\theta}{2a^2}\left (\left (1+\frac{ax+b}{\sqrt{n}} \right )  \right )^2\frac{\left ( f\left(x+\frac{2}{a\sqrt{n}}\right)-2f\left(x\right)+f\left(x-\frac{2}{a\sqrt{n}}\right)  \right )}{\left (\frac{2}{a\sqrt{n}}\right )^2}\\
        &\longrightarrow -\theta\left(x+\frac{b}{a}\right)f'\left(x\right)+\frac{1}{2}\frac{\theta}{a^2}f''\left( x\right), \,\,\, n \to \infty.
          \end{align*}

Comparing the above limit with \eqref{OUgenerator} we see that
\begin{equation}\label{Ou_gener_limit}
A_n f(x) \longrightarrow \mathcal{A} f(x), \,\,\, n \to \infty
\end{equation}
with $\tau=\theta$, $\mu=-\frac{b}{a}$ and $\sigma=\frac{1}{a\sqrt{2}}$.

Since the function $f$ is in the space $C^{3}_{c}(\R)$, the above convergence holds uniformly. Furthermore $$x'_n \rightarrow x, \,\, n \rightarrow \infty,$$ so we have
$$\underset{x \in \R}{sup}\, |A_n f(x'_{n})-A_n f(x)| \rightarrow 0, \,\,\, n \to \infty.$$

Using the above together with \eqref{Ou_gener_limit} in \eqref{OU_gen_conver} we finally obtain
$$||A_n f_n-\mathcal{A} f||_{\infty} \rightarrow 0, \quad n \to \infty,$$
and since  $$X^{(n)}_{0}\Rightarrow X_{0}, \quad n \to \infty \quad \iff \quad x'_n \to x, \quad n \to \infty,$$
by Theorem \ref{chain_aprox} we obtain $X^{n} \Rightarrow X$ in $\mathbb{D}(\R)$, where $X$ is the generally parametrized OU diffusion.
\end{proof}

\section{Correlated CTRWs for non-heavy-tailed Pearson diffusions based on the Wright-Fisher model} \label{CTRWforFPD_WFUS}

In this section we present two different versions of the Wright-Fisher model that lead to generally parametrized Jacobi and CIR diffusions.

\subsection{Jacobi diffusion} \label{FracJacobiDiff}
Jacobi diffusion $Y = \left(Y_t, \,t \geq 0 \right)$ is defined as the solution of the SDE
\begin{equation*}
dY_t = -\gamma (Y_t - \mu) dt + \sqrt{2 \gamma \delta Y_t (1 - Y_t)} dW_t, \quad t \geq 0, \quad \mu \in (0,1), \quad \gamma, \delta > 0,
\end{equation*}
with the infinitesimal generator
\begin{equation}\label{Jacgenerator}
\mathcal{A} f(y) = -\gamma (y - \mu) f'(y) + \frac{1}{2} (2\gamma \delta) y(1-y) f''(y), \quad f \in C^{3}_{c}([0,1]).
\end{equation}

In this case we assume that there are no survival abilities and that probability of mutation of each type ($A$ to $a$ and $a$ to $A$) is proportional to the size of the population, i.e.
\begin{equation*}
  \alpha = \frac{a}{n}, \quad \beta = \frac{b}{n}, \quad s = 0, \quad a, b > 0.
\end{equation*}

Therefore, the expected fraction of mature $A$-types \eqref{WF_pi} becomes
\begin{equation}\label{Jac_pi}
  p_i = \frac{i}{n} \left(1 - \frac{a}{n} \right) + \left(1 - \frac{i}{n} \right) \frac{b}{n}.
\end{equation}

For each $n \in \mathbb{N}$ we define a new Markov chain $(H^{(n)}_{r}, \, r \in \mathbb{N}_{0})$ as follows:
\begin{equation}\label{new_chain_Jac}
H^{(n)}_{r} = \frac{1}{n} G^{(n)}_{r}.
\end{equation}
The state space of this new Markov chain is $\{0, \frac{1}{n}, \frac{2}{n}, \ldots, 1\}$ and we assume that the initial state of the original chain $(G^{(n)}_{r}, \, r \in \mathbb{N}_{0})$ is given by $$i(y) =i= G_{0}^{n} = \lfloor{n y \rfloor}, \quad y \in [0,1].$$ We emphasize that the initial state is a function of $y$, but we will use the notation $i$ for simplicity. The transition operator $T_n$ of the Markov chain $(H^{(n)}_{r}, \, r \in \mathbb{N}_{0})$ is defined as follows:
\begin{equation}\label{trans_op_Jac}
T_{n}f\left(\frac{i}{n}\right)=\sum_{j=0}^{n}p_{ij}f\left(\frac{j}{n}\right),
\end{equation}
where $p_{ij}$ is defined in \eqref{WF_pij} and $p_i$ in \eqref{Jac_pi}.

Define the operator
\begin{equation}\label{CTRW_gen_Jac}
A_n := \theta n(T_n-I), \quad \theta >0, \quad f_n \in Dom(A_n), \quad f_n(y):=f\left(\frac{\lfloor{n y\rfloor}}{n}\right)=f\left(\frac{i}{n}\right),
\end{equation}
where $f \in C_{c}^{3}\left([0,1]\right)$.

Now by the following scaling of time in $(H^{(n)}_{r}, \, r \in \mathbb{N}_{0})$, for each $n \in \mathbb{N}$ we obtain the corresponding continuous-time process $(Y^{(n)}_t, \, t \geq 0)$:
\begin{equation}\label{TimeChange_Jac}
Y^{(n)}_t := H^{(n)}_{\lfloor \theta nt \rfloor}, \quad \theta >0.
\end{equation}

The next theorem states that the Jacobi diffusion could be obtained as the limiting process of the previously time-changed processes $(Y^{(n)}_t, \, t \geq 0)$.

\begin{theorem} \label{Thm_JacDiffusion}
Let $(H^{(n)}_{r}, \, r \in \mathbb{N}_{0})$, for each $n \in \mathbb{N}$, be the Markov chain defined by \eqref{new_chain_Jac} with the transition operator \eqref{trans_op_Jac}.  Let $Y^{n} = (Y^{(n)}_t, \, t \geq 0)$, for each $n \in \mathbb{N}$, be its corresponding time-changed process with the time-change  \eqref{TimeChange_Jac}. Let the operators $(A_{n}, \, n \in \mathbb{N})$ be defined by \eqref{CTRW_gen_Jac}. Then
\begin{equation*}
  Y^{n} \Rightarrow Y \,\,\text{in}\,\, \mathbb{D}([0,1]),
\end{equation*}
where $Y = (Y_t, \, t \geq 0)$ is the Jacobi diffusion with the infinitesimal generator $\mathcal{A}$ given by \eqref{Jacgenerator}, and
\begin{equation*}
  \delta = \frac{1}{2(a+b)}, \quad \gamma = \theta(a+b), \quad \mu = \frac{b}{a+b}.
\end{equation*}
\end{theorem}

\begin{proof}
Once again, we are in the setting of the Theorem \ref{chain_aprox} and follow the same procedure as in the Theorem \ref{Thm_OUDiffusion}. We first prove that statement $a)$ of Theorem \ref{chain_aprox} is valid for $A_n$ defined by \eqref{CTRW_gen_Jac} and $\mathcal{A}$ defined by \eqref{Jacgenerator}. Then we use the equivalence of statements $a)$ and $d)$ from Theorem \ref{chain_aprox} to obtain convergence $Y^{n} \Rightarrow Y$ in $\mathbb{D}([0,1])$ under assumption $Y^{(n)}_{0}\Rightarrow Y_{0}$, $n \to \infty$.

First, we prove the statement $a)$ from Theorem \ref{chain_aprox} in our setting:

$$||f_n-f||_{\infty}=\underset{y \in [0,1]}{sup}\, |f_n\left(y\right)-f\left(y\right)|=\underset{y \in [0,1]}{sup}\, |f\left(i/n\right)-f\left(y\right)|.$$
According to the well-known property $\lfloor n y \rfloor \leq n y< \lfloor n y \rfloor+1$ of the function $\lfloor \cdot \rfloor$, it follows that
\begin{equation*}
  \frac{i}{n} \leq  y< \frac{i}{n}+\frac{1}{n}
\end{equation*}
and therefore
\begin{equation*}
  \lim_{n \to \infty}\frac{i}{n}=y.
\end{equation*}
From this we obtain
$$||f_n-f||_{\infty} \rightarrow 0, \,\, n \rightarrow \infty$$
and
$$||A_n f_n-\mathcal{A} f||_{\infty}=\underset{y \in [0,1]}{sup} \,|A_n f_n(y)-\mathcal{A} f(y)|=\underset{y \in [0,1]}{sup}\, |A_n f(i/n)-\mathcal{A} f(y)|.$$

According to \eqref{CTRW_gen_Jac} it follows that
\begin{equation*}
  A_n f\left(\frac{i}{n}\right)= \theta n \left(\sum_{j=0}^{n}p_{ij}f\left(\frac{j}{n}\right)-f\left(\frac{i}{n}\right)\right)=\theta n \sum_{j=0}^{n}p_{ij}\left(f\left(\frac{j}{n}\right)-f\left(\frac{i}{n}\right)\right).
\end{equation*}

By the Taylor formula for function $f$ around $\frac{i}{n}$ with the mean-value form of the remainder  we obtain
\begin{equation}\label{gen_Jac}
A_n f\left(\frac{i}{n}\right)=\theta \sum_{j=0}^{n}p_{ij}\left(j-i\right)f'\left(\frac{i}{n}\right)+\theta \sum_{j=0}^{n}p_{ij}\frac{\left(j-i\right)^2}{2n}f''\left(\frac{i}{n}\right)+\theta \sum_{j=0}^{n}p_{ij}\frac{\left(j-i\right)^3}{6n^2}f'''\left(\zeta\right) ,
\end{equation}
where $\zeta$ is a real number such that $\lvert\zeta- \frac{i}{n}\rvert < \lvert \frac{j}{n}-\frac{i}{n}\rvert$. \\
Next, denote
\begin{equation*}
  \mu(y):=\lim_{n \to \infty}\sum\limits_{j=0}^{n}p_{ij}\left(j-i\right), \quad \sigma^2(y):=\lim_{n \to \infty}\sum\limits_{j=0}^{n}p_{ij}\frac{\left(j-i\right)^2}{n}, \quad R_n(y):=\theta \sum_{j=0}^{n}p_{ij}\frac{\left(j-i\right)^3}{6n^2}f'''\left(\zeta\right).
\end{equation*}

Taking into account \eqref{gen_Jac}, we obtain
\begin{equation}\label{gen_Jac_lim}
\lim_{n \to \infty} A_n f\left(\frac{i}{n}\right)=\theta \mu(y)f'\left(y\right)+\frac{\theta}{2} \sigma^{2}(y)f''\left(y\right)+\lim_{n \to \infty} R_n(y).
\end{equation}

From the time-homogeneity of the observed Markov chains and the fact that $f \in C^{3}_{c}([0,1])$, it follows that
\begin{align}\mu(y)&=\lim_{n \to \infty}\sum\limits_{j=0}^{n}p_{ij}\left(j-i\right)=\lim_{n \to \infty}E[G^{n}_1-G^{n}_0 | G^{n}_0=i]\nonumber\\
                     &=\lim_{n \to \infty} \left(n p_i -i \right)=\lim_{n \to \infty} \left(i\left(1-\frac{a}{n}\right)+\left(1-\frac{i}{n}\right)b-i\right)\nonumber\\
                     &=-(a+b)\lim_{n \to \infty}\left(\frac{i}{n}\right)+b=-(a+b)y+b, \label{mu_sig_jac1}
\end{align}
\begin{align}
         \sigma^2(y)&=\lim_{n \to \infty}\sum\limits_{j=0}^{n}p_{ij}\frac{\left(j-i\right)^2}{n}=\lim_{n \to \infty}\frac{1}{n}E[\left(G^{n}_1-G^{n}_0\right)^2 | G^{n}_0=i]\nonumber\\
                     &=\lim_{n \to \infty} \frac{1}{n}\left(n p_i(1-p_i)+n^2 p^2_{i}-2n i p_i +i^2 \right)=\lim_{n \to \infty} \left(p_i-p^2_i+\frac{\left(n p_i-i\right)^2}{n}\right)\nonumber\\
                     &=y-y^2=y(1-y), \label{mu_sig_jac2}
\end{align}
\begin{align}
         \lim_{n \to \infty}|R_n(y)| &\leq   K \bigg \lvert \theta \lim_{n \to \infty} \sum_{j=0}^{n}p_{ij}\frac{\left(j-i\right)^3}{6n^2}\bigg\rvert= K \bigg \lvert\theta \lim_{n \to \infty}\frac{1}{n^2}E[\left(G^{n}_1-G^{n}_0\right)^3 | G^{n}_0=i]\bigg \lvert\nonumber\\
                     &= K \bigg \lvert\theta\lim_{n \to \infty} \frac{1}{n^2}\left(n p_i(1-3p_i+2p^{2}_{i})+3n p_{i}(n p_i -i)(1-p_i)+(np_i-i)^3 \right)\bigg \lvert\nonumber \\
                     &= K \bigg \lvert\theta \lim_{n \to \infty} \left(\frac{p_i(1-3p_i+2p^{2}_{i})}{n}+\frac{3 p_{i}(n p_i -i)(1-p_i)}{n}+\frac{(np_i-i)^3}{n^2} \right)\bigg \lvert\nonumber\\
                     &=0, \label{Rn_jac2} \nonumber\\
\end{align}
where $K$ is a constant such that $|f'''(\zeta)| \leq K$ and
\begin{equation*}
  \lim_{n \to \infty} p_i = y \quad \text{and} \quad \lim_{n \to \infty} \frac{\left(n p_i-i\right)^2}{n}=0.
\end{equation*}

Finally, by substituting \eqref{mu_sig_jac1}, \eqref{mu_sig_jac2} and \eqref{Rn_jac2} in \eqref{gen_Jac_lim}, we obtain
\begin{equation*}
  \lim_{n \to \infty} A_n f\left(\frac{i}{n}\right)=\theta(-y(a+b)+b)f'(y)+\frac{1}{2}\theta y(1-y) f''(y).
\end{equation*}

Now by re-parametrizing
\begin{equation*}
  \delta = \frac{1}{2(a+b)}, \quad \gamma = \theta(a+b), \quad \mu = \frac{b}{a+b},
\end{equation*}
the last limit becomes
\begin{equation*}
  \lim_{n \to \infty} A_n f\left(\frac{i}{n}\right)=-\gamma(y-\mu)f'(y)+\frac{1}{2}(2\gamma\delta)y(1-y) f''(y),
\end{equation*}
which is precisely the infinitesimal generator of the Jacobi diffusion.
In the space $C^{3}_{c}([0,1])$ all above limits hold uniformly, therefore we obtain
$$||A_n f_n-\mathcal{A} f||_{\infty} \to 0, \quad n \to \infty,$$
and since  $$Y^{(n)}_{0}\Rightarrow Y_{0}, \quad n \to \infty \quad \iff \quad i/n \to y, \quad n \to \infty,$$
by Theorem \ref{chain_aprox} we obtain $Y^{n} \Rightarrow Y$ in $\mathbb{D}([0,1])$, where $Y$ is the generally parametrized Jacobi diffusion.

\end{proof}

\subsection{CIR process} \label{FracCIRDiff}

The CIR diffusion $Z = (Z_t, t \geq 0)$ is defined as the solution of the SDE
\begin{equation*}
dZ_t=-\theta \left(Z_t -\frac{b}{a} \right)dt+\sqrt{\frac{\theta}{a}Z_t}dW_t, \quad t\geq 0,\quad \theta >0,\quad a>0,\quad b>0,
\end{equation*}
with the infinitesimal generator
\begin{equation}\label{CIRgenerator}
\mathcal{A} f(z)=-\theta \left(z -\frac{b}{a} \right)f'(z)+\frac{1}{2}\frac{\theta}{a}z f''(z),\quad f \in C^{3}_{c}([0,\infty)).
\end{equation}
In this case, we assume there is only mutation of the order
$$\alpha=\frac{a}{n^{d}},\quad \beta=\frac{b}{n},\quad 0<d<1,\quad 0<a,\,\,b<\infty,\quad s=0,$$
so that expected fraction of $A$-types \eqref{WF_pi} becomes
\begin{equation}\label{CIR_pi}
p_i=\frac{i}{n}\left (1-\frac{a}{n^d} \right)+\left(1-\frac{i}{n} \right)\frac{b}{n}.
\end{equation}
For each $n \in \N$, we define the new Markov chain $(H^{(n)}_r, \,\, r \in \N)$ with the state space $\{0,\,\frac{1}{n^d},\dots, \frac{1}{n^{d-1}}\}$
\begin{equation}\label{new_chain_CIR}
H^{(n)}_r=\frac{G^{(n)}_r}{n^{d}}.
\end{equation}
We assume that the initial state space of the starting Markov chain $(G^{(n)}_r, \,\, r \in \mathbb{N}_{0})$ is
$$i(z)=i=G^{n}_0=\lfloor n^{d} z \rfloor,\,\, z \in [0,\,\infty ).$$
We assume that $n$ is always large enough so that $i(z)$ is in the state space of Markov chain $(G^{(n)}_r, \,\, r \in \mathbb{N}_{0})$. We emphasize that the initial state is a function of $z$, but we will use notation $i$ for simplicity. The transition operator $T_n$ of the Markov chain $(H^{(n)}_{r}, \,\, n \in \N)$ is given by
\begin{equation}\label{trans_op_CIR}
T_{n}f\left(\frac{i}{n}\right)=\sum_{j=0}^{n}p_{ij}f\left(\frac{j}{n}\right),
\end{equation}
where $p_{ij}$ is defined in \eqref{WF_pij} and $p_i$ in \eqref{CIR_pi}.
Now we define operator
\begin{equation}\label{CTRW_gen_CIR}
A_n: = \frac{\theta}{a} n(T_n-I),\quad\theta >0,\quad f_n \in Dom(A_n), \quad f_n(z):=f\left(\frac{\lfloor{n^d z\rfloor}}{n^d}\right)=f\left(\frac{i}{n^d}\right)
\end{equation}
where $f \in C_{c}^{3}\left([0,\infty)\right)$
and by the following scaling of time in $(H^{(n)}_{r}, \, r \in \mathbb{N}_{0}),$ for each $n \in \N$ we obtain the corresponding continuous-time process $(Z^{(n)}_t, \, t \geq 0)$:
\begin{equation}\label{TimeChange_CIR}
Z^{(n)}_t : = H^{(n)}_{\lfloor \frac{\theta }{a}n^d t \rfloor}, \quad \theta >0.
\end{equation}
The next theorem states that the CIR diffusion could be obtained as the limiting process of the time-changed processes $(Z^{(n)}_t, \, t \geq 0)$.

\begin{theorem} \label{Thm_CIRDiffusion}
Let $(H^{(n)}_{r}, \, r \in \mathbb{N}_{0})$, for each $n \in \mathbb{N}$, be the Markov chain defined by \eqref{new_chain_CIR} with the transition operator \eqref{trans_op_CIR}. Let $Z^{n} = (Z^{(n)}_t, \, t \geq 0)$, for each $n \in \mathbb{N}$, be its corresponding time-changed process, with the time-change  \eqref{TimeChange_CIR}. Let the operators $(A_{n}, \, n \in \mathbb{N})$ be defined by \eqref{CTRW_gen_CIR}.  Then
\begin{equation*}
  Z^{n} \Rightarrow Z \,\,\text{in}\,\, \mathbb{D}(\R^+),
\end{equation*}
where $Z = (Z_t, \, t \geq 0)$ is the CIR diffusion with the infinitesimal generator $\mathcal{A}$ given by \eqref{CIRgenerator}.
\end{theorem}

\begin{proof}

 We first prove that statement $a)$ of the Theorem \ref{chain_aprox} is valid for $A_n$ defined by \eqref{CTRW_gen_CIR} and $\mathcal{A}$ defined by \eqref{CIRgenerator} and then we use the equivalence of statements $a)$ and $d)$ from Theorem \ref{chain_aprox} to obtain convergence $Z^{n} \Rightarrow Z$ in $\mathbb{D}(\R^+)$ under assumption $Z^{(n)}_{0}\Rightarrow Z_{0}$, $n \to \infty$.

We first prove the statement $a)$ from Theorem \ref{chain_aprox} in our setting:

$$||f_n-f||_{\infty}=\underset{z \in \R^{+}}{sup}\, |f_n\left(z\right)-f\left(z\right)|=\underset{z \in \R^{+}}{sup}\, |f\left(i/n^{d}\right)-f\left(z\right)|.$$
By definition of the function $\lfloor \cdot \rfloor$ we have
$$\lfloor n^d z \rfloor \leq n^d z< \lfloor n^d z \rfloor+1$$
so it follows
$$\frac{i}{n^d} \leq  z< \frac{i}{n^d}+\frac{1}{n^d}$$
and therefore
$$\lim_{n \to \infty}\frac{i}{n^d}=z.$$
so using this we obtain
$$||f_n-f||_{\infty} \rightarrow 0, \,\, n \rightarrow \infty$$
and
$$||A_n f_n-\mathcal{A} f||_{\infty}=\underset{z \in \R^{+}}{sup} \,|A_n f_n(z)-\mathcal{A} f(z)|=\underset{z \in \R^{+}}{sup}\, |A_n f(i/n^{d})-\mathcal{A} f(z)|.$$

According to \eqref{CTRW_gen_CIR} it follows
$$A_n f\left(\frac{i}{n^d}\right)= \frac{\theta }{a}n^d \left(\sum_{j=0}^{n}p_{ij}f\left(\frac{j}{n^d}\right)-f\left(\frac{i}{n^d}\right)\right)=\frac{\theta }{a}n^d \sum_{j=0}^{n}p_{ij}\left(f\left(\frac{j}{n^d}\right)-f\left(\frac{i}{n^d}\right)\right). $$
Now by the Taylor formula for function $f$ around $\frac{i}{n^d}$ with mean-value form of the remainder  we obtain
\begin{equation}\label{gen_CIR}
A_n f\left(\frac{i}{n^n}\right)=\frac{\theta}{a} \sum_{j=0}^{n}p_{ij}\left(j-i\right)f'\left(\frac{i}{n^d}\right)+\frac{\theta}{2a} \sum_{j=0}^{n}p_{ij}\frac{\left(j-i\right)^2}{n^d}f''\left(\frac{i}{n^d}\right)+\frac{\theta}{6a} \sum_{j=0}^{n}p_{ij}\frac{\left(j-i\right)^3}{n^{2d}}f'''\left(\zeta\right)
\end{equation}
where $\zeta$ is a real number such that $\lvert\zeta- \frac{i}{n^d}\rvert < \lvert \frac{j}{n^d}-\frac{i}{n^d}\rvert$. \\
Let us denote$$\mu(z):=\lim_{n \to \infty}\sum\limits_{j=0}^{n}p_{ij}\left(j-i\right),\quad \sigma^2(z):=\lim_{n \to \infty}\sum\limits_{j=0}^{n}p_{ij}\frac{\left(j-i\right)^2}{n^d}, \quad R_n(z):=\frac{\theta}{6a} \sum_{j=0}^{n}p_{ij}\frac{\left(j-i\right)^3}{n^{2d}}f'''\left(\zeta\right).$$
Together with \eqref{gen_CIR} we obtain
\begin{equation}\label{gen_CIR_lim}
\lim_{n \to \infty} A_n f\left(\frac{i}{n^d}\right)=\frac{\theta}{a} \mu(z)f'\left(z\right)+\frac{\theta}{2a} \sigma^{2}(z)f''\left(z\right)+\lim_{n \to \infty} R_n(z).
\end{equation}
From the time-homogeneity of the Markov chains, it follows that
\begin{align}\label{mu_sig_CIR}\mu(z)&=\lim_{n \to \infty}\sum\limits_{j=0}^{n}p_{ij}\left(j-i\right)=\lim_{n \to \infty}E[G^{n}_1-G^{n}_0 | G^{n}_0=i]\nonumber\\
                     &=\lim_{n \to \infty} \left(n p_i -i \right)=\lim_{n \to \infty} \left(i\left(1-\frac{a}{n^d}\right)+\left(1-\frac{i}{n}\right)b-i\right)\nonumber\\
                     &=-a\lim_{n \to \infty}\left(\frac{i}{n^d}\right)-\lim_{n \to \infty}\left(\frac{i}{n}\right)b+b=-az+b,
\end{align}
\begin{align}\label{mu_sig_CIR_2}
         \sigma^2(z)&=\lim_{n \to \infty}\sum\limits_{j=0}^{n}p_{ij}\frac{\left(j-i\right)^2}{n^d}=\lim_{n \to \infty}\frac{1}{n^d}E[\left(G^{n}_1-G^{n}_0\right)^2 | G^{n}_0=i]\nonumber\\
                     &=\lim_{n \to \infty}\frac{1}{n^d} \left(n p_i(1-p_i)+n^2 p^2_{i}-2n i p_i +i^2 \right)\nonumber\\
                     &=\lim_{n \to \infty}\frac{1}{n^d} \left(n p_i(1-p_i)+(n p_i -i)^2\right)\nonumber\\
                     &=z,
\end{align}
\begin{align}
         \lim_{n \to \infty}|R_n(z)| &\leq   K \bigg \lvert \frac{\theta}{6 a} \lim_{n \to \infty} \sum_{j=0}^{n}p_{ij}\frac{\left(j-i\right)^3}{n^{2d}}\bigg\rvert= K \bigg \lvert \frac{\theta}{6 a} \lim_{n \to \infty}\frac{1}{n^{2d}}E[\left(G^{n}_1-G^{n}_0\right)^3 | G^{n}_0=i]\bigg \lvert\nonumber\\
                     &= K \bigg \lvert\frac{\theta}{6 a}\lim_{n \to \infty} \frac{1}{n^{2d}}\left(n p_i(1-3p_i+2p^{2}_{i})+3n p_{i}(n p_i -i)(1-p_i)+(np_i-i)^3 \right)\bigg \lvert\nonumber \\
                     &= K \bigg \lvert\frac{\theta}{6 a} \lim_{n \to \infty} \left(\frac{n p_i}{n^d}\frac{(1-3p_i+2p^{2}_{i})}{n^d}+\frac{3 n p_{i}}{n^d}\frac{(n p_i -i)(1-p_i)}{n^d}+\frac{(np_i-i)^3}{n^{2d}} \right)\bigg \lvert\nonumber\\
                     &=0, \label{Rn_CIR2} \nonumber\\
\end{align}
where $K$ is a constant such that $|f'''(\zeta)| \leq K$ and
$$\frac{n p_i}{n^d} \rightarrow z\,\,\,\, p_i \rightarrow 0,\,\,\,\, \frac{(n p_i -i)^2}{n^d}\rightarrow 0,\,\,\,\, n \to \infty.$$
Finally, substituting \eqref{mu_sig_CIR}, \eqref{mu_sig_CIR_2} and \eqref{Rn_CIR2}  in \eqref{gen_CIR_lim} we obtain
\begin{align*}
\lim_{n \to \infty} A_n f\left(\frac{i}{n^d}\right)&=\frac{\theta}{a}\left(-az+b\right)f'(z)+\frac{\theta}{2 a} z f''(z)\\
                                                   &=-\theta\left(z-\frac{b}{a}\right)f'(z)+\frac{1}{2}\frac{\theta}{a} z f''(z).
\end{align*}

Comparing the obtained limit with \eqref{CIRgenerator} we see that the limit coincides with the generator of the CIR diffusion.
Since  $f \in C^{3}_{c}([0,\infty))$, all above limits hold uniformly, i.e., we  obtain
$$||A_n f_n-\mathcal{A} f||_{\infty} \to 0, \quad n \to \infty,$$
and since $$Z^{(n)}_{0}\Rightarrow Z_{0}, \quad n \to \infty \quad \iff \quad i/n^{d} \to z, \quad n \to \infty,$$
by Theorem \ref{chain_aprox} we obtain $Z^{n} \Rightarrow Z$ in $\mathbb{D}(\R^+)$, where $Z$ is the generally parametrized CIR diffusion.

\end{proof}

\section{Fractional OU, Jacobi and CIR diffusions as the correlated CTRWs limits}

Suppose that $(T_{r}, \, r \in \mathbb{N}_{0})$, where $T_{0}=0$, $T_{r}=G_{1}+\ldots+G_{r}$, is the random walk where $G_{r} \geq 0$ are iid waiting times between particle jumps that are independent of the Markov chain $(H^{(n)}_{r}, \, r \in \mathbb{N}_{0})$. We assume $G_{1}$ is in the domain of attraction of the $\beta$-stable distribution with index $0 < \beta <1$, and that the waiting time of the Markov chain until its $r$-th move is described by $G_r$.   Let \begin{equation}\label{renewal}
N_{t} = \max\{r \geq 0 \colon T_{r} \leq t\}
\end{equation}
be the number of jumps up to time $t \geq 0$. Then the continuous time stochastic process
\begin{equation*}
\left( H^{(n)}(N_{t}), \, t \geq 0 \right),
\end{equation*}
where $H^{(n)}(N_{t})$ is the state of the Markov chain at time $t \geq 0$,
 is the correlated CTRW process. We now state the main result of this section by defining the correlated CTRWs that converge to the fractional OU, CIR, and Jacobi diffusions.

\begin{theorem}\label{CTRWconv}
Let $(H^{(n)}_{r}, \, r \in \mathbb{N}_{0})$ be the Markov chain defined by \eqref{new_chain_OU} in the case of OU diffusion, by \eqref{new_chain_Jac} in the case of Jacobi diffusion and by \eqref{new_chain_CIR} in the case of CIR diffusion. Let $(X^{(n)}(t), \, t \geq 0), \quad (Y^{(n)}(t), \, t \geq 0)$ and $(Z^{(n)}(t), \, t \geq 0)$ be the corresponding rescaled Markov chains given by \eqref{TimeChange_OU}, \eqref{TimeChange_Jac} and \eqref{TimeChange_CIR} respectively.
Let $(N(t), \, t \geq 0)$ be the renewal process defined in \eqref{renewal}, and $(E_{t}, \, t \geq 0)$ be the inverse of the standard $\beta$-stable subordinator $(D(t), \, t \geq 0)$ with $0 < \beta < 1$. Then
\begin{align*}
&X^{(n)}(n^{-1} N(n^{\frac{1}{\beta}}t)) \Rightarrow X(E_t),\quad n \to \infty,\\
&Y^{(n)}(n^{-1} N(n^{\frac{1}{\beta}}t)) \Rightarrow Y(E_t),\quad n \to \infty, \\
&Z^{(n)}(n^{-1}N(n^{\frac{1}{\beta}}t)) \Rightarrow Z(E_t),\quad n \to \infty,
\end{align*}
in the Skorokhod space $\mathbb{D}(S)$ with $J_1$ topology, where $\left ( X(t), \, t \geq 0 \right )$ is Ornstein-Uhlenbeck diffusion with generator \eqref{OUgenerator}, $\left ( Y(t), \, t \geq 0 \right )$ is Jacobi diffusion with generator \eqref{Jacgenerator} and $\left ( Z(t), \, t \geq 0 \right )$ is  CIR diffusion with generator \eqref{CIRgenerator}.
\end{theorem}

\begin{proof}

From \cite[Section 4.4]{MeerschaertSikorskii_2011}
$$n^{-\frac{1}{\beta}} T(\lceil nt \rceil) \Rightarrow D_t, \quad n \to \infty$$
in the sense of finite dimensional distributions, where $(D_t, \,t \geq 0)$ is $\beta$-stable subordinator.
Since $\beta$-stable subordinator $D_t$ is a L\'evy process, it follows that $D_t$ is continuous in probability. Since the sample paths of the process
$T(\lceil n t \rceil )$ are monotone non-decreasing, \cite[Theorem 3]{Bingham_1971} yields
\begin{equation}\label{J1conv_Et}
n^{-\frac{1}{\beta}} T(\lceil nt \rceil) \Rightarrow D_t, \quad n \to \infty
\end{equation}
in the Skorokhod space $\mathbb{D}(\R^+)$ with $J_1$ topology. From \eqref{J1conv_Et} and Theorems \ref{Thm_OUDiffusion}, \ref{Thm_JacDiffusion} and \ref{Thm_CIRDiffusion} we obtain the joint convergence
\begin{align} \label{AtEtjointconvergence}
  &\left( X^{(n)}(t), n^{-\frac{1}{\beta}} T(\lceil nt \rceil) \right) \Rightarrow \left( X_{t}, D_{t} \right), \quad n \to \infty,\\
  &\left( Y^{(n)}(t), n^{-\frac{1}{\beta}} T(\lceil nt \rceil) \right) \Rightarrow \left( Y_{t}, D_{t} \right), \quad n \to \infty,\\
  &\left( Z^{(n)}(t), n^{-\frac{1}{\beta}} T(\lceil nt \rceil) \right) \Rightarrow \left( Z_{t}, D_{t} \right), \quad n \to \infty,
\end{align}
in the product space $\mathbb{D}(S \times \R^+)$ with $J_1$ topology, where $S=\R$ in Ornstein-Uhlenbeck diffusion case, $S=[0,1]$ in Jacobi diffusion case and $S=\R^+$ in CIR diffusion case.
Following the notation from \cite{Henry_Straka_2011} let
$$\alpha = (\beta, \sigma) \in \mathbb{D} (S \times \R^{+}), \,\, \beta \in \mathbb{D} (S),\, \sigma \in \mathbb{D} (\R^+),$$
and $D_u, D_{\uparrow}$ and $D_{\uparrow \uparrow}$ be sets of all such $\alpha$ which have unbounded, non-decreasing and increasing $\sigma$, respectively.
As shown in \cite{Henry_Straka_2011} sets
$$D_{\uparrow, u}=D_{\uparrow} \cap D_{u},\,\,D_{\uparrow \uparrow, u}=D_{\uparrow \uparrow} \cap D_{u}$$
are Borel measurable.

Introduce the function
$$\Psi : D_{\uparrow, u} \mapsto \mathbb{D} (S \times \R^{+}), \,\, \Psi(\alpha)=\beta \circ \sigma^{-1}.
$$
From \cite[Proposition 2.3]{Henry_Straka_2011} function $\Psi$ is continuous in $D_{\uparrow \uparrow, u}.$

Note that $\left( X^{(n)}(t), n^{-\frac{1}{\beta}} T(\lceil nt \rceil) \right), \, \, \left( X_{t}, D_{t} \right) $ are in the domain of function $\Psi$, i.e. $$\left( X^{(n)}(t), n^{-\frac{1}{\beta}} T(\lceil nt \rceil) \right), \, \, \left( X_{t}, D_{t} \right) \in D_{\uparrow, u}.$$
Also, since the standard $\beta$-stable  subordinator $(D_{t},\,t\geq 0)$ is strictly increasing, it follows
\begin{equation}\label{sub_inc}
\left( X_{t}, D_{t} \right) \in D_{\uparrow \uparrow, u}.
\end{equation}
Observe that for the generalized inverses we have
$$(n^{-\frac{1}{\beta}} T(\lceil nt \rceil))^{-1}=n^{-1} N(n^{\frac{1}{\beta}}t), \,\,\, (D_t)^{-1}=E_t.$$
Since the function $\Psi$ is continuous at $D_{\uparrow \uparrow, u}$,
 $$\Psi\left( X^{(n)}(t),\, n^{-\frac{1}{\beta}} T(\lceil nt \rceil) \right)=  X^{(n)}(n^{-1} N(n^{\frac{1}{\beta}}t))$$
 and
 $$\Psi\left ( X(t), D(t) \right )=X(E_t),$$
 using \eqref{AtEtjointconvergence} and \eqref{sub_inc} it follows
$$ X^{(n)}(n^{-1} N(n^{\frac{1}{\beta}}t)) \Rightarrow  X(E_t)$$
in the space $\mathbb{D}(\R)$ with  $J_1$ topology.
The proofs for convergence of the respective CTRWs to fractional CIR or Jacobi diffusion are the same as in the case of convergence to fractional OU diffusion.
\end{proof}
\begin{remark}
The proof of Theorem \ref{CTRWconv} uses the approach from the proof of \cite[Theorem 3.6]{Henry_Straka_2011} but in our case the first component is Markov chain, rather then a random walk.
\end{remark}
\begin{remark}
Note that
\begin{align*}
X^{(n)}(n^{-1} N(n^{\frac{1}{\beta}}t)) &= H^{(n)}\left(\lfloor 2^{-1}\theta N(n^{\frac{1}{\beta}}t)\rfloor \right),\\
Y^{(n)}(n^{-1} N(n^{\frac{1}{\beta}}t)) &= H^{(n)}\left(\lfloor \theta N(n^{\frac{1}{\beta}}t) \rfloor\right),\\
Z^{(n)}(n^{-1} N(n^{\frac{1}{\beta}}t)) &= H^{(n)}\left(\lfloor a^{-1}\theta n^{d-1} N(n^{\frac{1}{\beta}}t) \rfloor \right),
\end{align*}
where $(H^{(n)}_{r}, \, r \in \mathbb{N}_{0})$ is corresponding Markov chain given by \eqref{new_chain_OU} for the OU, \eqref{new_chain_Jac} for Jacobi and by  \eqref{new_chain_CIR} for CIR case.
\end{remark}
\bigskip \bigskip

\textbf{Acknowledgements} \newline

The authors wish to thank the referee for the comments and suggestions which have led to
the improvement of the manuscript.

N.N. Leonenko was supported in particular by Cardiff Incoming Visiting Fellowship
Scheme and International Collaboration Seedcorn Fund, Cardiff  Data Innovation Research Institute Seed Corn Funding, Australian Research
Council's Discovery Projects funding scheme (project number DP160101366), and by
projects MTM2012-32674 and MTM2015--71839--P (co-funded with Federal funds), of the
DGI, MINECO, Spain.

N. \v{S}uvak and I. Papi\'c were supported by the project IZIP-2016 funded by the J.J. Strossmayer University of Osijek. I. Papi\'c was also supported by the Almae Matris Alumni Croaticae UK (AMAC UK), Association of Alumni and Friends of Croatian Universities in the United Kingdom.

\newpage

\bibliographystyle{abrv}
\bibliography{CTRWS}

\end{document}